\documentclass[13pt,a4paper]{amsart}
\usepackage{amsthm}
\usepackage{mathrsfs}
\usepackage{CJK}
\usepackage{amsfonts}
\usepackage{xy}
\usepackage{cite}
\usepackage{amssymb}
\usepackage{color, latexsym}
\usepackage{graphicx}
\usepackage{amsbsy,amscd}
\usepackage{fancyhdr}
\usepackage[absolute]{textpos}
\usepackage[colorlinks=true,
           citecolor=blue,
           linkcolor =red]{hyperref}
\makeatletter

\xyoption{all}
\allowdisplaybreaks

\fontsize{9pt}{9pt}\selectfont
\def\ga{\alpha}

\def\gl{\lambda}
\def\bgl{\boldsymbol{\lambda}}
\def\sh{\mathscr{H}}
\def\mpn{\mathscr{P}(m, n)}
\def\rr{\rightarrow}
\def\la{\langle}
\def\ra{\rangle}
\newenvironment{point}[2]%
  {\ifx*#2\let\pointlabel\relax\else\def\pointlabel{#2}\fi
   \refstepcounter{equation}\trivlist
   \item[\hskip\labelsep\theequation.
         \ifx\pointlabel\relax\else\space\pointlabel\space\fi]
   \ignorespaces #1
  }{\relax}
\numberwithin{equation}{section}
\swapnumbers
\newtheorem{theorem}[equation]{Theorem}
\newtheorem{lemma}[equation]{Lemma}

\newtheorem{corollary}[equation]{Corollary}

\theoremstyle{definition}
\newtheorem{definition}[equation]{Definition}

\theoremstyle{remark}
\newtheorem{remark}[equation]{Remark}
\begin{document}
\begin{CJK*}{GBK}{song}
\setlength{\itemsep}{-0.2cm}
\fontsize{11}{\baselineskip}\selectfont
\setlength{\parskip}{0.35\baselineskip}
\vspace*{-5mm}
\title[Symbolical and cancellation-free formulae]{\fontsize{10}{\baselineskip}\selectfont The Symbolical and cancellation-free formulae for Schur elements }
\author{Deke Zhao}
\address{\begin{tabular}{l}
          \textsl{School of Applied Mathematics, Beijing Normal University at Zhuhai, Zhuhai, 519085, China}\\
         %\textsl{Academy of Mathematics and Systems Science, Chinese Academy of Sciences, Beijing 100190, China}\\
           \textsl{E-mail address: deke@amss.ac.cn}
           \end{tabular}}
\thanks{This research was supported by NSFC grant no.~11101037.}
\begin{abstract}In this paper we give the symbolical formula and cancellation-free formula for the Schur elements
 associated to the simple modules of the degenerate cyclotomic Hecke algebras. As some applications, we show
  that the Schur elements are symmetric polynomials with rational integer coefficients and give a different proof of
  Ariki-Mathas-Rui's criterion on the semisimplicity of the degenerate cyclotomic Hecke algebras.
\end{abstract}
\maketitle
%\tableofcontents
%%%%%%%%%%%%%%%%%%%%%%%%%%%%%%%%%%%%%%%%%%%%%%%%%%%%%%%%%%%%%%%%%%%%
%%%%%%%%%%%%%%%%%%%%%%%%%%%%%%%%%%%%%%%%%%%%%%%%%%%%%%%%%%%%%%%%%%%%
\vspace*{-10mm}
\section{Introduction}
%%%%%%%%%%%%%%%%%%%%%%%%%%%%%%%%%%%%%%%%%%%%%%%%%%%%%%%%%%%%%%
%%%%%%%%%%%%%%%%%%%%%%%%%%%%%%%%%%%%%%%%%%%%%%%%%%%%%%%%%%%%%%
 Schur elements play a powerful role in the representation theory of symmetric algebras (see e.g.
\cite[Chap.\,9]{Curtis-Reiner} and \cite[Chap.\,7]{Geck}). In the case of the degenerate
 cyclotomic Hecke algebra (dCHA), Brundan and Kleshchev \cite[Theorem A2]{BK-Schur-Weyl} showed
 that it is a symmetric algebra for all parameters, which enables us to use the Schur
 elements to determine when Specht modules of the dCHA are projective irreducible
 and whether the algebra is semisimple.

  Very recently, an explicit formula for
 the Schur elements associated to the simple modules of dCHA was given by
  the author in \cite{zhao} following Mathas' work \cite{Mathas-J-Algebra}. This paper continues our study on the Schur elements, which is inspired by Geck, Iancu and Malle's work \cite{GIM}, Mathas' work \cite{Mathas-J-Algebra} and Chlouveraki and Jacon's work
  \cite{Ch-Jacon}. The aim of this paper is to give an $L$-symbolical formula for the Schur elements (Theorem~\ref{Them:symmetric}) and a cancellation-free formula for the Schur elements (Theorem~\ref{Thm:cancellation}).
   As some direct applications, we show that the Schur elements are symmetric polynomials with rational
    integer coefficients  and provide a different proof of  Ariki-Mathas-Rui's criterion on the semisimplicity
    of the dCHA.

 The lay-out of this paper as follows. In section~\ref{Sec:Preliminaries} we introduce the necessary definitions and fix the notation.
  The $L$-symbolical formula for the Schur elements is given in
    Section~\ref{Sec:symmetric} and then the cancellation-free formula for the Schur elements is determined in Section~\ref{Sec cancelation}. Finally, some direct applications
      of the formulae and some remarks are given in Section~\ref{Sec: Applications}.
\subsection*{Acknowledgments}This work was partially carried out while the author was visiting the Academy of
Mathematics and Systems Science, CAS in Beijing. We are most deeply indebted to
Nanhua Xi and Yang Han for their invaluable help. We are grateful to Ming Fang for useful conversations.

%%%%%%%%%%%%%%%%%%%%%%%%%%%%%%%%%%%%%%%%%%%%%%%%%%%%%%%%%%%%%%%%%%%%
%%%%%%%%%%%%%%%%%%%%%%%%%%%%%%%%%%%%%%%%%%%%%%%%%%%%%%%%%%%%%%%%%%%%
\section{Preliminaries}\label{Sec:Preliminaries}
%%%%%%%%%%%%%%%%%%%%%%%%%%%%%%%%%%%%%%%%%%%%%%%%%%%%%%%%%%%%%%
%%%%%%%%%%%%%%%%%%%%%%%%%%%%%%%%%%%%%%%%%%%%%%%%%%%%%%%%%%%%%%
In this section, we introduce the necessary definitions and notation.

\begin{point}{}*A {\it partition} $\lambda=(\lambda_1\geq\lambda_2\geq\cdots)$ is a decreasing sequence of non-negative integers containing only finitely many non-zero terms.
We define the {\it length} of $\lambda$ to be the smallest integer $\ell(\lambda)$
such that $\lambda_i=0$ for all $i>\ell(\lambda)$ and set $|\lambda|:=\sum_{i \geq 1}\lambda_i$.
If $|\lambda|=n$ we say that $\lambda$ is a \textit{partition} of $n$.

The \textit{diagram} of a partition $\lambda$ may be formally defined as the set
$[\lambda]:=\{(i,j)\,\mid\, i\geq 1, 1 \leq j \leq \lambda_i\}$.
The elements of $[\gl]$ are the \textit{nodes} of $\gl$ and we say that a node $(i,j)$ of $\gl$ is
{\it removable} if $[\lambda]\backslash\{(i,j)\}$ is still the diagram of a partition of $|\gl|-1$.

 The \textit{conjugate} of  a partition $\gl$ is the partition $\hat{\gl}=(\hat{\gl}_1\geq\hat{\gl}_2\geq\dots)$
 whose  diagram is the transpose of the diagram of $\gl$, i.e. $\hat{\gl}_i$ is the number of nodes in the $i$th column of the diagram of $\gl$. Hence $\hat{\gl}_1=\ell(\gl)$ and a node
$(i,j)$ of $\gl$ is removable if and only if $j=\gl_i$ and $i=\hat{\gl}_{j}$.

Recall that the $(i,j)$-th {\it hook} in the diagram $[\lambda]$ is the
collection of nodes to the right of and below the node $(i,j)$,
including the node $(i,j)$ itself, and that the $(i,j)$-th {\it hook length}
$h^{\lambda}_{i,j}=\lambda_i-i+\hat{\lambda}_j-j+1$ is the
number of nodes in the $(i,j)$-th hook.
 \end{point}

We will need the following lemma, whose proof is an easy combinatorial exercise.

\begin{lemma}\label{Lemm:indermine-identity}Let $\mu=(\mu_1\geq\mu_2\geq\dots \geq \mu_k>\mu_{k+1}=0)$ be a partition  and let $y$ be an indeterminate. Then for any integer $\ell$ such that $1\le \ell\le \mu_1$, we have
$$ \frac{1}{\mu_1+y}\prod_{1\le i\le\hat{\mu}_{\ell}}\frac{\mu_i-i+1+y}{\mu_i-i+y}=\frac{1}{\ell-\hat{\mu}_{\ell}-1+y}
\prod_{\ell\le j\le \mu_1}\frac{j-\hat{\mu}_j-1+y}{j-\hat{\mu}_j+y}.$$
\end{lemma}

\begin{point}{}* Let $m, n$ be positive integers. Recall from \cite{shephard-toda} or \cite{cohen} that the complex
  reflection group $W_{m,n}$ of type $G(m, 1, n)$ is the finite group generated by elements $s_0, s_1, \dots, s_{n-1}$ subject to the relations
 $$\begin{aligned}&s_0^m=1, \quad
 s_0s_1s_0s_1=s_1s_0s_1s_0\\
 &s_i^2=1, \quad s_is_{i+1}s_i=s_{i+1}s_{i}s_{i+1}, \quad i\geq1\\
 & s_is_j=s_js_i, \quad |i-j|>1.\end{aligned}$$
 In particular, the subgroup $\la s_1, \dots, s_{n-1}\ra$ of $W_{m,n}$ is isomorphic to the
 symmetric group $S_n$ of degree
 $n$ with simple transpositions $\sigma_i=(i,i+1)$ for $i=1, \dots, n-1$.
It is well-known that $W_{m,n}\cong(\mathbb{Z}/m\mathbb{Z})^n\rtimes S_n$. Clearly, $W_{1,n}$ is
the Weyl group of type $A_n$ and
 $W_{2,n}$ is the Weyl group of type $B_n$.
 \end{point}

\begin{definition}\label{Def DCHA}Let $R$ be a field and $Q=(q_1, \dots, q_m)\in R^m$.
The {\it degenerate cyclotomic Hecke algebra} (dCHA) is the unital
associative $R$-algebra  $\sh(Q)\!:=\sh_{m,n}(Q)$ with generators
$s_0,s_1,\dots,s_{n-1}$ and relations

\begin{enumerate}\item
$(s_0-q_1)\dots(s_0-q_m)=0$,
\item $s_0(s_1s_0s_1+s_1)=(s_1s_0s_1+s_1)s_0$,
\item $s_i^2=1$,                \quad $1\le i<n$,
\item $s_is_{i+1}s_i=s_{i+1}s_is_{i+1}$, \quad  $1\le i<n-1$,
\item $s_is_j=s_js_i$,                   \quad$|i-j|>1$.
\end{enumerate}
\end{definition}

\begin{remark}\begin{enumerate}\item Definition~\ref{Def DCHA}
coincides with that given in \cite[\S 7]{Kbook}) and that given in \cite[P$_{61}$]{AMR}.
 \item $\sh_{1,n}(Q)$ is exactly the group algebra $RS_n$ and $\sh(Q)$ is a (degenerate cyclotomic) deformation
of the $R$-group algebra of the complex reflection group $W_{m,n}$.\end{enumerate} \end{remark}

\begin{point}{}*
The elements $x_1\!:=\!s_0$
and $x_{i+1}\!:=\!s_ix_is_i\!+\!s_i$, $1\!\le\!i\!\le\!n\!-\!1$, are called
the \textit{Jucys-Murphy elements} of $\sh(Q)$. In \cite[Theorem 7.5.6]{Kbook},
Kleshchev proved that $\sh$ is a
free $R$-module with basis \begin{align*}&\{x_1^{i_1}x_2^{i_2}\cdots x_n^{i_n}w\mid0\le i_1, \dots, i_n<m, w\in S_n\}.\end{align*}
Let $\tau:\sh(Q)\rr R$ be the $R$-linear map determined by
$$\tau(x_1^{i_1}\cdots x_n^{i_n} w)
        :=\left\{\begin{array}{ll} 1,&\text{ if }i_1=\cdots=i_n=m-1 \text{ and } w=1,\\
                0,&\text{ otherwise}.
                \end{array}\right.$$
Then $\tau$ is a non-degenerate trace form on $\sh(Q)$ for all $Q\in R^m$, i.e. $\sh(Q)$  is a
symmetric algebra for all $Q\in R^m$, see \cite[Theorem A2]{BK-Schur-Weyl}.\end{point}

Recall that an {\it $m$-multipartition} of $n$ is a ordered $m$-tuple
$\bgl=(\gl^1; \cdots; \gl^m)$ of partitions $\lambda^{i}$ such that
$n=\sum_{i=1}^m|\lambda^{i}|$. We define the {\it length} of $\bgl$ to be
$\ell(\bgl)=\max\{\ell(\lambda^{s})|1\le s\le m\}$ and denote by $\mpn$ the set of all $m$-multipartitions of $n$.
The {\it diagram} of an $m$-multipartition $\boldsymbol{\lambda}$ is the set
\begin{align*}[\bgl]:=\{(i,j,c)\in\mathbb{Z}_{>0}\times\mathbb{Z}_{>0}\times \mathbf{m}|1\le j\le\lambda^c_i\} \quad\text{ where }\mathbf{m}=\{1, \dots, m\}.\end{align*}
As  in $[\bgl]$ the element is called the \textit{node} of $\bgl$; more generally, a node is any element of
$\mathbb{Z}_{>0}\times \mathbb{Z}_{>0}\times \mathbf{m}$. We  may and will identify $[\bgl]$ with
 the $m$-tuple of diagrams of the partitions $\lambda^{c}$, for $1\le c\le m$.

\begin{point}{}*\label{separated} Let $R$ be a field and $Q=(q_1, \dots, q_m)\in R^m$. We set
 \begin{align*}P_\sh(Q):=n!
  \prod_{1\le i<j\le m}\prod_{|d|<n}(d+q_i-q_j).\end{align*}
  Ariki, Mathas and Rui~\cite[Theorem 6.11]{AMR} have shown
that $\sh(Q)$ is semisimple if and only if $P_\sh(Q)\!\neq\!0$. This criterion will be recovered form our results
later (see Theorem~\ref{ssimple}).

From now on we assume that $\sh(Q)$ is semisimple.
Then $\{S^{\bgl}\mid\bgl \in \mpn\}$ is a complete set of pairwise
non-isomorphic irreducible $\sh(Q)$-modules and we let $\chi^{\bgl}$ be the
character of $S^{\bgl}$. Following Geck's results on symmetrizing forms (see \cite[Theorem 7.2.6]{Geck}),
we obtain the following definition for the Schur elements of $\sh(Q)$ associated to the irreducible
representations of $\sh(Q)$. \end{point}

\begin{definition}
Assume that $\sh(Q)$ is semisimple. The {\it Schur
elements} of $\sh(Q)$ are the elements $s_{\bgl}(Q)\in R$ such that
$$\tau=\sum_{\bgl\in\mpn}\frac{\chi^{\bgl}}{s_{\bgl}(Q)}.$$
\end{definition}

An explicit formula for the Schur elements for $\sh(Q)$ is given by the following theorem.
\begin{theorem}[\cite{zhao}, Theorem~7.9]Let $\bgl=(\gl^1;\dots;\gl^m)$ be an $m$-multipartition of $n$. Then
\begin{align*}s_{\bgl}(Q)=\prod_{(i,j,s)\in[\bgl]}\!\!%
    h^{\lambda^{s}}_{ij}\prod_{1\le s<t\le m}X^{\bgl}_{st},\end{align*}
where, for $1\le s<t\le m$,
\begin{align*}X^{\bgl}_{st}=\prod_{(i,j)\in[\gl^t]}(j-i+q_t-q_s)\prod_{(i,j)\in [\gl^s]}\biggl((j-i-\gl_1^t+q_s-q_t)
\prod_{1\le k\le\gl_1^t}\frac{j-i+\hat{\gl}_k^t-k+1+q_s-q_t}{j-i+\hat{\gl}_k^t-k+q_s-q_t}\biggr).\end{align*}
\label{Them:Schur-elements}
 \vspace{-5mm}
\end{theorem}

\begin{remark}The formula for the Schur elements for the dCHA can not be obtained from
the ones for non-degenerate cyclotomic Hecke algebra by specializing $q$ to $1$.\end{remark}
 %%%%%%%%%%%%%%%%%%%%%%%%%%%%%%%%%%%%%%%%%%%%%%%%%%%%%%%%%%%%%%%%%%%%%%%%%
 %%%%%%%%%%%%%%%%%%%%%%%%%%%%%%%%%%%%%%%%%%%%%%%%%%%%%%%%%%%%%%%%%%%%%%%%%%
\section{The $L$-symbolical formula for Schur elements}\label{Sec:symmetric}
%%%%%%%%%%%%%%%%%%%%%%%%%%%%%%%%%%%%%%%%%%%%%%%%%%%%%%%%%%%%%%%%%%%%%%%%%
%%%%%%%%%%%%%%%%%%%%%%%%%%%%%%%%%%%%%%%%%%%%%%%%%%%%%%%%%%%%%%%%%%%%%%%%%
In this section we give the $L$-symbolical formula for the Schur elements.
Before doing this, we need the notation of the $L$-symbol of a multipartition introduced by Malle \cite{Malle95}.

\begin{definition}Let $\gl=(\gl_1, \cdots, \gl_t)$ be a partition and fix an integer $L$ such that $L\ge\ell(\lambda)$. The
$L$-{\it beta numbers} for $\lambda$ are the integers
$\beta^{\gl}_i=\lambda_i+L-i$ for $i=1,\dots,L$. For an $m$-multipartition $\bgl=(\gl^1;\dots;\gl^m)$,
 the $m\times L$ matrix $B^{\bgl}_L=(\beta^{s}_i)_{s,i}$, where $\beta^{s}_i=\gl_i^s+L-i$, is called the $L$-{\it symbol} of $\bgl$.
\end{definition}

From now on, we denote by $B^{\gl}_L$ the set of the $L$-beta numbers for a partition $\lambda$. Observe that
 \begin{align*}B^{\gl}_L= \big\{\beta_i^\gl=\gl_i+L-i\mid 1\le i\le\ell(\gl)\big\}\cup \big\{L-i\mid \ell(\gl)+1\leq i
 \leq L\big\}\end{align*} and that if we change $L$ to $L+1$ then $B_L^\gl$ is {\it shifted} to
$B_{L+1}^\gl=\big\{k+1\mid k\in B^\gl_{L}\big\}\cup \big\{0\big\}$. We say that a function
of beta numbers is {\it invariant under beta shifts} if it is unchanged by
such transformations; equivalently, the function is independent of
$L$ provided that $L$ is large enough. For example, the formula for
$s_\lambda(Q)$ is invariant under beta shifts since
$s_\lambda(Q)$ does not depend on $L$.

The following lemma is the key to the $L$-symbolical formula for the Schur elements.
\begin{lemma}Let $\gl$, $\mu$ be partitions and let $x$ be an
indeterminate. Define
\begin{align*}
X_{\gl\mu}(x)&:=\prod_{(i,j)\in[\mu]}(j-i-x)\prod_{(i,j)\in [\gl]}\biggl((j-i-\mu_1+x)
\prod_{1\le k\le\mu_1}\frac{j-i+\hat{\mu}_k-k+1+x}{j-i+\hat{\mu}_k-k+x}\biggr)\quad\text{ and }\\ \vspace{1\jot}
Y^{L}_{\gl\mu}(x)&:=(-1)^{\binom L2}x^L\,\frac{\displaystyle\prod_{a\in B_L^\gl}\prod_{1\leq i\leq a}(i+x)
\prod_{b\in B_L^\mu}\prod_{1\le j\le b}(j-x)}{\displaystyle\prod_{(a,b)\in B_L^\gl\!\times\!B_L^\mu}(a-b+x)}.
\end{align*}
Then  $X_{\gl\mu}(x)=Y^L_{\gl\mu}(x)$  for any integer $L\geq \max\big\{\ell(\gl), \ell(\mu)\big\}$.
\label{Lemm:X=Y}
\end{lemma}

\begin{proof}First, we show that $Y^{L}_{\gl\mu}(x)$ is invariant under beta shifts, i.e.
 $Y^{L}_{\gl\mu}(x)=Y^{L+1}_{\gl\mu}(x)$ for all integers
$L\geq \max\big\{\ell(\gl), \ell(\mu)\big\}$. Note that $B_{L+1}^\gl=\big\{i+1\mid i\in
B^\gl_L\big\}\cup\big\{0\big\}$ and $B^\mu_{L+1}=\big\{i+1\mid i\in B^\mu_L\big\}\cup\big\{0\big\}$.
 Therefore
\begin{align*}Y^{L+1}_{\gl\mu}(x)&=(\!-\!1\!)^{\binom {L+1}2}x^{L+1}\,\frac{\displaystyle\prod_{a\in B_{L+1}^\gl}
\displaystyle\prod_{1\leq i\leq a}(i+x)
\displaystyle\prod_{b\in B_{L+1}^\mu}\displaystyle\prod_{1\le j\le b}(j-x)}{\displaystyle\prod_{(a,b)\in B_{L+1}^\gl\!\times\!B_{L+1}^\mu}(a-b+x)}\\
      &=(\!-\!1\!)^{\binom {L+1}2}x^{L}\,
\frac{\displaystyle\prod_{a\in B_{L}^\gl}\displaystyle\prod_{1\leq i\leq a+1}(i+x)
\displaystyle\prod_{b\in B_{L}^\mu}\displaystyle\prod_{1\le j\le b+1}(j-x)}
{\displaystyle\prod_{(a,b)\in B_{L}^\gl\!\times\!B_{L}^\mu}(a-b+x)
\prod_{a\in B_{L}^\gl}(a+1+x)\prod_{b\in B_{L}^\mu}(x-b-1)}\\
      &=(\!-\!1\!)^{\binom {L+1}2+L}x^{L}\,\frac{\displaystyle\prod_{a\in B_{L}^\gl}\prod_{1\leq i\leq a}(i+x)
\prod_{b\in B_{L}^\mu}\prod_{1\le j\le b}(j-x)}{\displaystyle\prod_{(a,b)\in B_{L}^\gl\!\times\!B_{L}^\mu}(a-b+x)}\\
      &=Y^{L}_{\gl\mu}(x).\end{align*}

As a consequence, we may choose $L$ arbitrarily, provided that
$L\ge\max\big\{\ell(\lambda), \ell(\mu)\big\}$. Now we proceed to prove the lemma
 by induction on the numbers of nodes of the partitions $\gl$ and $\mu$. Our start step is
 that $X_{\gl\mu}(x)\!=\!Y^L_{\gl\mu}(x)=1$ when $\lambda\!=\!\mu\!=\!(0)$, which follows directly by taking $L\!=\!0$.

Next, assume by way of induction that we have
proved the lemma for all partitions $\nu$ and $\mu$ with $|\nu|\leq k-1\geq 0$. We consider
the partitions $\gl$ with $|\gl|=k\geq 1$ and $\mu$. Note that we
may choose $L$ large enough for partitions $\gl$, $\mu$ and $\nu$. Since $|\gl|\geq 1$, there exists $\imath$
 such that $(\imath, \jmath)$ is a removable node of $[\gl]$. Abusing notation,
  we denote by $\nu$
 the partition obtained from $\gl$ by removing the node $(\imath, \jmath)$, that is
 $[\nu]=[\gl]\backslash\big\{(\imath, \jmath)\big\}$.
 Then  \begin{align*}&B^\mu_{L}=
  \big\{\mu_i+L-i\mid 1\leq i\leq \ell(\mu)\big\}\cup \big\{L-i\mid \ell(\mu)+1\leq i\leq L \big\},\\
  &B^\gl_L=\big(\big\{\gl_i+L-i\mid 1\leq i\leq \ell(\nu)\big\}\cup \big\{L-i\mid \ell(\nu)+1\leq i\leq L \big\}
  \cup \big\{\jmath+L-\imath\big\}\big)\backslash\big\{\jmath-1+L-\imath\big\}.\end{align*}
  Thus, we obtain that
  \begin{align*}Y^{L}_{\gl\mu}(x)&=(-1)^{\binom L2}x^L\frac{\displaystyle\prod_{a\in B_L^\gl}\prod_{1\leq i\leq a}(i+x)
  \prod_{b\in B_L^\mu}\prod_{1\le j\le b}(j-x)}{\displaystyle\prod_{(a,b)\in B_L^\gl\!\times\!B_L^\mu}(a-b+x)}\\
      &=Y^{L}_{\nu\mu}(x)(\jmath+L-\imath+x)\prod_{b\in B^{\mu}_L}\frac{b+1+\imath-\jmath-L-x}{b+\imath-\jmath-L-x}\\
      &=Y^{L}_{\nu\mu}(x)(\jmath+L-\imath+x)\prod_{1\le i\le
       \ell(\mu)}\frac{\mu_i-i+1+\imath-\jmath-x}{\mu_i-i+\imath-\jmath-x}\prod_{\ell(\mu)+1\le i\le
        L}\frac{i-1+\jmath-\imath+x}{i+\jmath-\imath+x}\\
      &=Y^{L}_{\nu\mu}(x)(\jmath+\ell(\mu)-\imath+x)\prod_{1\le i\le
      \ell(\mu)}\frac{\mu_i-i+1+\imath-\jmath-x}{\mu_i-i+\imath-\jmath-x}.
   \end{align*}
  On the other hand, we have
\begin{align*}X_{\gl\mu}(x)&=X_{\nu\mu}(x)(\jmath-\imath-\mu_1+x)
\prod_{1\le k\le\mu_1}\frac{\hat{\mu}_k-k+1+\jmath-\imath+x}{\hat{\mu}_k-k+\jmath-\imath+x}\\
&=X_{\nu\mu}(x)(\jmath-\imath-\mu_1+x)
\prod_{1\le k\le\mu_1}\frac{k-\hat{\mu}_k-1+\imath-\jmath-x}{k-\hat{\mu}_k+\imath-\jmath-x}\\
&=X_{\nu\mu}(x)(\jmath+\ell(\mu)-\imath+x)\prod_{1\le i\le
 \hat{\mu}_1}\frac{\mu_i-i+1+\imath-\jmath-x}{\mu_i-i+\imath-\jmath-x}\\
&=X_{\nu\mu}(x)(\jmath+\ell(\mu)-\imath+x)\prod_{1\le i\le
\ell(\mu)}\frac{\mu_i-i+1+\imath-\jmath-x}{\mu_i-i+\imath-\jmath-x},\end{align*}
where the third equality follows by applying Lemma~\ref{Lemm:indermine-identity} for $\ell=1$,
$y=\imath-\jmath-x$ and the lats equality follows by noting that $\hat{\mu}_1=\ell(\mu)$.
This proves the lemma for the partitions $|\gl|$ with $|\gl|=k$ and $\mu$ by applying the induction argument.

Finally, assume by way of induction that we have
proved the lemma for all partitions $\gl$ and $\nu$ with $|\nu|=r-1\geq0$.
We consider the partitions $\gl$ and $\mu$ with $|\mu|=r$. Similarly we can choose $L$ large enough for $\gl$, $\mu$ and
$\nu$.  Since $|\mu|\geq 1$, there exists $\imath$
 such that $(\imath, \jmath)$ is a removable node of $[\mu]$. Abusing  notation,
  we denote by $\nu$ the partition obtained form $\mu$ by removing the node
  $(\imath, \jmath)$, that is $[\nu]=[\mu]\backslash \big\{(\imath, \jmath)\big\}$.
 Then \begin{align*}&B^\gl_{L}=\big\{\gl_i+L-i\mid 1\leq i\leq \ell(\gl)\big\}\cup \big\{L-i\mid
 \ell(\gl)+1\leq i\leq L \big\},\\
 &B^\mu_L=\big(\big\{\mu_i+L-i\mid 1\leq i\leq \ell(\nu)\big\}\cup \big\{L-i\mid \ell(\nu)+1\leq i\leq L \big\}
 \cup \big\{\jmath+L-\imath\big\}\big)\backslash\big\{\jmath-1+L-\imath\big\}.
 \end{align*}
  So, we obtain that
     \begin{align*}Y^{L}_{\gl\mu}(x)&=(-1)^{\binom L2}x^L\frac{\displaystyle\prod_{a\in B_L^\gl}
     \prod_{1\leq i\leq a}(i+x)
\prod_{b\in B_L^\mu}\prod_{1\le j\le b}(j-x)}{\displaystyle\prod_{(a,b)\in B_L^\gl\!\times\!B_L^\mu}(a-b+x)}\\
      &=Y^{L}_{\gl\nu}(x)(\jmath+L-\imath-x)\prod_{a\in B^{\gl}_L}\frac{a+1+\imath-\jmath-L+x}{a+\imath-\jmath-L+x}\\
      &=Y^{L}_{\nu\mu}(x)(\jmath+L-\imath-x)\prod_{1\le i\le \ell(\gl)}\frac{\gl_i-i+1+\imath-\jmath+x}{\gl_i-i+\imath-\jmath+x}\prod_{\ell(\gl)+1\le k\le L}\frac{k-1+\jmath-\imath-x}{k+\jmath-\imath-x}\\
      &=Y^{L}_{\nu\mu}(x)(\jmath+\ell(\gl)-\imath-x)\prod_{1\le i\le \ell(\gl)}\frac{\gl_i-i+1+\imath-\jmath+x}{\gl_i-i+\imath-\jmath+x}.
   \end{align*}
On the other hand, if $\imath=1$ then $\jmath=\mu_1$ and that we have
\begin{align*}X_{\gl\mu}(x)&=X_{\gl\nu}(x)(\jmath-\imath-x)\prod_{(i,j)\in [\gl]}
\frac{j-i-\mu_1+x}{j-i+1-\mu_1+x}\,\,\frac{j-i+\mu_1+2+x}{j-i+\mu_1+1+x}\\
&=X_{\gl\nu}(x)(\jmath-\imath-\mu_1+x)\prod_{1\le i\le \ell(\gl)}\prod_{1-i\le j\le \gl_i-i}
\frac{j-\mu_1+x}{j+1-\mu_1+x}\,\,\frac{j+2+\mu_1+x}{j+1+\mu_1+x}\\
&=X_{\gl\nu}(x)(\jmath-\imath-\mu_1+x)\prod_{1\le i\le \ell(\gl)}
\frac{i-1+\mu_1-x}{i-2+\mu_1-x}\prod_{1\le i\le \ell(\gl)}\frac{\gl_i-i+2-\mu_1+x}{\gl_i-i+1-\mu_1+x}\\
&=X_{\gl\nu}(x)(\jmath+\ell(\gl)-1-x)\prod_{1\le i\le \ell(\gl)}
\frac{\gl_i-i+2-\mu_1+x}{\gl_i-i+1-\mu_1+x},\end{align*}
where the last equality follows by using Lemma~\ref{Lemm:indermine-identity} for $\ell=1$ and $y=-\mu_1+1+x$.

Assume that $\imath\neq 1$. Then we get that
\begin{align*}X_{\gl\mu}(x)&=X_{\gl\nu}(\jmath-\imath-x)\prod_{(i,j)\in [\gl]}
\frac{j-i-1+\imath-\jmath+x}{j-i+\imath-\jmath+x}\,\,\frac{j-i+1+\imath-\jmath+x}{j-i+\imath-\jmath+x}\\
&=X_{\gl\nu}(x)(\jmath-\imath+x)\prod_{1\le i\le \ell(\gl)}\prod_{1-i\le j\le \gl_i-i}
\frac{j-1+\imath-\jmath+x}{j+\imath-\jmath+x}\,\,\frac{j+1+\imath-\jmath+x}{j+\imath-\jmath+x}\\
&=X_{\gl\nu}(x)(\jmath-\imath+x)\prod_{1\le i\le \ell(\gl)}
\frac{i+\jmath-\imath-x}{i-1+\jmath-\imath-x}\prod_{1\le i\le \ell(\gl)}\frac{\gl_i-i+1+\imath-\jmath+x}{\gl_i-i+\imath-\jmath+x}\\
&=X_{\gl\nu}(x)(\jmath-\imath+\ell(\gl)-x)\prod_{1\le i\le \ell(\gl)}
\frac{\gl_i-i+1+\imath-\jmath+x}{\gl_i-i+\imath-\jmath+x}.\end{align*}
Combining the above arguments, we prove the lemma for partitions  $\gl$ and $\mu$ with $|\mu|=r$. As a consequence, we prove the lemma for all partitions $\gl$ and $\mu$.
\end{proof}

The following property of $X_{\gl\mu}(x)$ will be used later, see Corollary~\ref{Cor:Schur-sym}.

\begin{corollary}\label{Cor:symmetric}Keep notation as in Lemma~\ref{Lemm:X=Y}. Then $X_{\gl\mu}(x)=X_{\mu\gl}(-x)$.
\end{corollary}
\begin{proof}By Lemma~\ref{Lemm:X=Y}, it suffices to show that $Y^L_{\gl\mu}(x)=Y^{L}_{\mu\gl}(-x)$ for any integer
 $L\geq\max\big\{\ell(\gl),\ell(\mu)\big\}$, which follows directly from the definition of $Y^L_{\gl\mu}(x)$
 given in Lemma~\ref{Lemm:X=Y}.\end{proof}

Now we can obtain the $L$-symbolical formula for the Schur elements.
\begin{theorem}
Let $\bgl=(\gl^1;\dots;\gl^m)$ be an $m$-multipartition of $n$ with
$L$-symbol $B_L^{\bgl}=(\beta^{s}_i)_{s,i}$ such that $L\ge\ell(\bgl)$.
Then
\begin{align*}s_{\bgl}(Q)=\frac{(-1)^{\binom m2\binom L2}\displaystyle\prod_{1\le s<t\le m}(q_s-q_t)^L
      \displaystyle\prod_{1\le s,t\le m}\prod_{\alpha_s\in B_L^s}
      \displaystyle\prod_{1\le k\le\alpha_s} (k+q_s-q_t)}
     {\displaystyle\prod_{1\le s< t\le m}
      \displaystyle\prod_{\substack{(\alpha_s,\alpha_t)\in B_L^s\times B_L^t}}
             (\alpha_s+q_s-\alpha_t-q_t)\prod_{1\le s\le m}\prod_{1\le i<j\le L}(\beta_i^s-\beta_j^s)}.\end{align*}
\label{Them:symmetric}
\end{theorem}
\begin{proof}
For an $m$-multipartion $\bgl$ of $n$ with
$L$-symbol $B_L^{\bgl}=(\beta^{s}_i)_{s,i}$ such that $L\ge\ell(\bgl)$, we have the following well-known fact,
see \cite[Examples I.1(4)]{Macdonald},
\begin{align*}\prod_{(i,j,s)\in[\bgl]}h^{\lambda^s}_{ij}
            =\prod_{1\le s\le m}\frac{\displaystyle\prod_{1\le i\le L}\beta_i^s!}
                {\displaystyle\prod_{1\le i<j\le L}(\beta_i^s-\beta_j^s)}.\end{align*}
 Now let \begin{align*}&\nu_\lambda={(-1)^{\binom m2\binom L2}\displaystyle\prod_{1\le s<t\le m}(q_s-q_t)^L
      \displaystyle\prod_{1\le s, t\le m}\prod_{\alpha_s\in B_s}
      \displaystyle\prod_{1\le k\le\alpha_s} (k+q_s-q_t)};\\ &\delta_{\gl}={\displaystyle\prod_{1\le s< t\le m}
      \displaystyle\prod_{\substack{(\alpha_s,\alpha_t)\in B_L^s\times B_L^t}}
             (\alpha_{s}-\ga_{t}+q_s-q_t)\prod_{1\le s\leq m}\prod_{1\le i<j\le L}(\beta_i^s-\beta_j^s)}.\end{align*}
  Then
 \begin{align*}\frac{\nu_\lambda}{\delta_{\gl}}&=\!\!
 \displaystyle\prod_{1\le s\le m}\frac{\displaystyle\prod_{\alpha_s\in B_L^s}\ga_s!}{\displaystyle\prod_{1\leq i<j
 \leq L}(\beta_i^s\!-\!\beta_j^s)}\displaystyle\prod_{1\leq s<t\leq m}(\!-\!1\!)^{\binom
 L2}(\!q_s\!-\!q_t\!)^L\frac{\prod_{1\le s\neq t\le m}\prod_{\ga_s\in B_L^s}\prod_{1\leq k\leq
 \ga_s}(\!k\!+\!q_{s}\!-\!q_{t}\!)}{
      \displaystyle\prod_{(\!\alpha_s,\ga_t\!)\in B_L^s\!\times\!B_L^t}(\!\ga_{s}\!-\!\ga_{t}\!+\!q_{s}\!-\!q_{t}\!)}
\\&=\prod_{(i,j,s)\in[\bgl]}h^{\lambda^s}_{i,j}\prod_{1\leq s<t\leq m} Y_{\gl^s\gl^t}^{L}(q_s-q_t).\end{align*}
 So, by Theorem~\ref{Them:Schur-elements}, the theorem follows immediately by applying
  Lemma~\ref{Lemm:X=Y} for $\gl=\gl^s$, $\mu=\gl^t$ and $x=q_s-q_t$. We have complete the proof.
\end{proof}

\begin{remark} Geck, Iancu
and Malle \cite{GIM} have used a clever specialization argument due to Orellana \cite{O}
to compute the Schur elements for the cyclotomic Hecke algebra using the Markov trace of the Hecke algebras
$H_q(S_n)$, which does not work for dCHA due to that $\tau(1)=0\neq 1$. It would be interesting
 to know whether there is a ``degenerate" version Markov trace for dCHA satisfying the similar properties
as that of Markov trace.\end{remark}
%%%%%%%%%%%%%%%%%%%%%%%%%%%%%%%%%%%%%%%%%%%%%%%%%%%%%%%%%%%%%%
%%%%%%%%%%%%%%%%%%%%%%%%%%%%%%%%%%%%%%%%%%%%%%%%%%%%%%%%%%%%%%
\section{A cancellation-free formula for Schur elements}\label{Sec cancelation}
%%%%%%%%%%%%%%%%%%%%%%%%%%%%%%%%%%%%%%%%%%%%%%%%%%%%%%%%%%%%%%
%%%%%%%%%%%%%%%%%%%%%%%%%%%%%%%%%%%%%%%%%%%%%%%%%%%%%%%%%%%%%%
In this section, we give a cancellation-free formula for the Schur elements.
 Let $\gl$ and $\mu$ be partitions.
If $(i, j)$ is a node of $[\gl]$ and
we define \textit{the generalized hook length} of the node $(i, j)$ with respect to $(\gl, \mu)$ to be the integer
$h_{i,j}^{\gl,\mu}=\gl_i-i+\hat{\mu}_j-j+1$. Observe that if $\gl=\mu$ then $h_{i,j}^{\gl,\mu}=h_{i,j}^{\gl}$.

The following lemma is crucial to the cancellation-free formula for the Schur elements.

\begin{lemma}Assume that $\gl$, $\mu$ are partitions and that $x$ is an
indeterminate. Define
\begin{align*}
Z_{\gl\mu}(x):=\prod_{(i,j) \in [\lambda]}
 (h_{i,j}^{\lambda,\mu}+x)\,
\prod_{(i,j) \in [\mu]}
 (h_{i,j}^{\mu,\lambda}-x).
\end{align*}
Then $X_{\gl\mu}(x)=Z_{\gl\mu}(x)$.
\label{Lemm:X=Z}
\end{lemma}
\begin{proof}Our proof will proceed by induction on the number of nodes of the partition $\lambda$. We do not need to do
the same for $\mu$ by noting that $Z_{\mu\gl}(x)=Z_{\gl\mu}(-x)$ and Corollary~\ref{Cor:symmetric}.

If $\lambda=(0)$ then
\begin{align*}
X_{\gl\mu}(x)
&=\prod_{(i,j) \in [\mu]}(j-i-x)\\
&=\prod_{1\leq i\leq \ell(\mu)} \prod_{1\leq t \leq \mu_i}(t-i-x)\\
&=\prod_{1\leq i\leq \ell(\mu)} \prod_{1\leq j \leq \mu_i}(\mu_{i}-i-j+1-x)\\
&=\prod_{(i,j) \in [\mu]}(h_{i,j}^{\mu, \gl}-x)\\
&=Z_{\gl\mu}(x),
\end{align*}
 where the third equality follows by setting $t=\mu_i-i-j$.

Now assume that the assertion holds for all partitions $\gl$ with $|\lambda|\leq k-1\geq 0$. We show that it also holds
for partitions $\gl$ with $|\lambda|=k \geq 1$. Note that in this case, there exists $\imath$ such that $(\imath,\jmath)$
is a removable node of $\lambda$, i.e $\jmath=\gl_{\imath}$ and $\imath=\hat{\gl}_{\jmath}$.
Let $\nu$ be the partition obtained from
$\gl$  by removing the removable node $(\imath,\jmath)$.  So $[\lambda]=[\nu] \cup \big\{(\imath,\jmath)\big\}$
 and
 \begin{align*}X_{\lambda\mu}(x)&=X_{\nu\mu}(x)(\jmath-\imath-\mu_1+x)\prod_{1\leq i \leq \mu_1}
\frac{\hat{\mu}_i-i+1+\jmath-\imath+x}{\hat{\mu}_i-i+\jmath-\imath+x}.\end{align*}

On the other hand, we have
\begin{align*}Z_{\gl\mu}(x)&=\!Z_{\nu\mu}(x)\frac{\displaystyle\prod_{(\imath,j)\in[\gl]}(h_{\imath,j}^{\gl,\mu}+x)}
{\displaystyle\prod_{(\imath,j)\in[\nu]}(h_{\imath,j}^{\nu,\mu}+x)}
 \prod_{(i,\jmath)\in[\mu]}\frac{(h_{i,\jmath}^{\mu,\gl}-x)}
{(h_{i,\jmath}^{\mu,\nu}-x)}\\
&=\!Z_{\nu\mu}(x)(\hat{\mu}_{\jmath}-\imath+1+x)\prod_{1\le i\le \jmath-1}\frac{\hat{\mu}_i-i+1+\jmath-\imath+x}
{\hat{\mu}_i-i+\jmath-\imath+x}
 \prod_{1\le i\le \hat{\mu}_{\jmath}}\frac{\mu_i-i+1+\hat{\gl}_{\jmath}-\jmath-x}
{\mu_i-i+\hat{\gl}_{\jmath}-\jmath-x}\\
&=\!Z_{\nu\mu}(x)(\hat{\mu}_{\jmath}-\imath+1+x)\biggl(\prod_{1\le i\le \jmath-1}\frac{\hat{\mu}_i-i+1+\jmath-\imath+x}
{\hat{\mu}_i-i+\jmath-\imath+x}\biggr)\biggl(
 \prod_{1\le i\le \hat{\mu}_{\jmath}}\frac{\mu_i-i+1+\imath-\jmath-x}
{\mu_i-i+\imath-\jmath-x}\biggr)\\
&=\!Z_{\nu\mu}(x)(\hat{\mu}_{\jmath}\!-\!\imath\!+\!1\!+\!x)\biggl(\prod_{1\le i\le \jmath-1}\frac{\hat{\mu}_i\!-\!i\!+\!1\!+\!\jmath\!-\!\imath\!+\!x}
{\hat{\mu}_i\!-\!i\!+\!\jmath\!-\!\imath\!+\!x}\biggr)\biggl(
 \frac{\jmath\!-\!\imath\!-\!\mu_1\!+\!x}{\hat{\mu}_{\jmath}\!-\!\imath\!+\!1\!+\!x}\prod_{\jmath\le i\le \mu_1}\frac{\hat{\mu}_i\!-\!i\!+\!1\!+\!\jmath\!-\!\imath\!+\!x}
{\hat{\mu}_i\!-\!i\!+\!\jmath\!-\!\imath\!+\!x}\biggr)\\
&=\!Z_{\nu\mu}(x)(\jmath\!-\!\imath\!-\mu_1\!+\!x)\prod_{1\le i\le \mu_1}\frac{\!\hat{\mu}_i\!-\!i\!+\!1\!+\!\jmath\!-\!\imath\!+\!x}
{\!\hat{\mu}_i\!-\!i\!+\!\jmath\!-\!\imath\!+\!x},
\end{align*}
where the fourth equality follows by applying
Lemma~\ref{Lemm:indermine-identity} for $\ell=\jmath$ and $y=\imath-\jmath-x$; because that both $X_{\gl\mu}(x)$
 and $Z_{\gl\mu}(x)$ are invariant under beta-shifts, without loss of generality, we may assume that
  $\mu_1$ are large enough. Thus $X_{\gl\mu}(x)=Z_{\gl\mu}(x)$ for partitions $\gl$ and $\mu$ with
  $|\gl|=k$. As a consequence, we have completed the proof.
 \end{proof}

The following is the cancellation-free formula for the Schur elements, which is also symmetric.
\begin{theorem}\label{Thm:cancellation}Let $\bgl=(\gl^1; \dots; \gl^m)$ be an $m$-multipartition of $n$. Then
\begin{align*}s_{\bgl}(Q)=\prod_{1\le s\le m}\prod_{(i,j)\in[\gl^s]}\prod_{1\le t\le m}(h^{\gl^s, \gl^t}_{i,j}+q_s-q_t)=\prod_{1\le s\le m}\prod_{(i,j)\in[\gl^s]}\biggl(h^{\gl^s}_{i,j}\prod_{1\le t\le m\, \&\, t\neq s}(h^{\gl^s, \gl^t}_{i,j}+q_s-q_t)\biggr).\end{align*}
\end{theorem}

\begin{proof}Applying Lemma~\ref{Lemm:X=Z} for $\gl=\gl^s$, $\mu=\gl^t$ and $x=q_s-q_t$,
$X_{st}^{\lambda}=Z_{\gl^s\gl^t}(q_s-q_t)$  for all $1\leq s<t\leq m$. Thus the theorem follows directly by using Theorem~\ref{Them:Schur-elements}.
\end{proof}

%%%%%%%%%%%%%%%%%%%%%%%%%%%%%%%%%%%%%%%%%%%%%%%%%%%%%%%%%%%%%%%%%%%%%%
%%%%%%%%%%%%%%%%%%%%%%%%%%%%%%%%%%%%%%%%%%%%%%%%%%%%%%%%%%%%%%%%%%%%%%%%%
\section{Applications}\label{Sec: Applications}
%%%%%%%%%%%%%%%%%%%%%%%%%%%%%%%%%%%%%%%%%%%%%%%%%%%%%%%%%%%%%%
%%%%%%%%%%%%%%%%%%%%%%%%%%%%%%%%%%%%%%%%%%%%%%%%%%%%%%%%%%%%%%
In this section we give several direct applications of Theorems~\ref{Them:symmetric} and  \ref{Thm:cancellation}
and some remarks.

\begin{point}{}* Let $S_m$ be the symmetric group of order $m$ with simple transpositions
$\sigma_i=(i,i+1)$ for $i=1, \dots, m-1$.
 Note that there is an action of $S_m$ on the set of $m$-multipartitions of $n$
(by permuting components) and also on the rational functions in
$q_1,\dots,q_m$ (by permuting parameters). As a direct application of Theorem~\ref{Them:symmetric},
we obtain the following symmetry formula for the Schur elements, which can also be obtained by observing that
the Specht modules are determined up to isomorphism by the action of Jucys-Murphy elements
$x_1,\dots,x_n$ of $\sh(Q)$ and that the relation $\prod_{i=1}^m(x_1-q_i)=0$ is
invariant under the $S_m$-action.
\end{point}
\begin{corollary}\label{Cor:Schur-sym}
 Assume that $R$ is a field and that $\sh(Q)$ is semi-simple. Then
$s_{\sigma(\bgl)}(Q)=\sigma(s_{\bgl}(Q))$ for all $m$-multipartitions
$\bgl$ of $n$ and all $\sigma\in S_m$.
\end{corollary}
\begin{proof}For $i=1,\dots, m-1$, by applying Corollary~\ref{Cor:symmetric},
 \begin{align*}X_{\sigma_i(\gl^{i}\gl^{i+1})}(q_i-q_{i+1})=X_{\gl^{i+1}\gl^{i}}(q_{i}-q_{i+1})
 =X_{\gl^{i}\gl^{i+1}}(q_{i+1}-q_{i})=X_{\gl^{i}\gl^{i+1}}(\sigma_i(q_{i}-q_{i+1})).\end{align*}
 By the proof of Theorem~\ref{Them:symmetric}, $s_{\bgl}(Q)=\prod_{(i,j,s)\in[\bgl]}h^{\gl^s}_{i,j}\prod_{1\leq s<t\leq m}X_{s, t}(q_s-q_t)$. So $s_{\sigma_i(\bgl)}(Q)=\sigma_i(s_{\bgl}(Q))$ for all $1 \le i\le m-1$. As a consequence we complete the proof. \end{proof}

As a direct application of the cancellation-free formula,  we obtain the following fact on the Schur elements, which can also be proved by a similar argument to that of \cite[Proposition~7.3.9]{Geck}.

\begin{corollary}Assume that $R$ is a field and that $\sh(Q)$ is semisimple. Then for all $m$-multipartitions $\bgl$
of $n$, $s_{\bgl}(Q)$ is a  polynomial in variables $q_1, \cdots, q_m$ with rational integer coefficients.
\label{Cor:integral}\end{corollary}

\begin{proof}The corollary follows directly by applying Theorem~\ref{Thm:cancellation}. \end{proof}

A second application of the cancellation-free formula is that we can easily recover a
well-known semisimplicity criterion for the degenerate  cyclotomic Hecke algebra due to Ariki, Mathas and Rui
\cite[Theorem 6.11]{AMR}.
To do this, let us assume that $q_1, \dots, q_{m}$ are indeterminates and $R=\mathbb{Q}(q_1, \dots, q_{m})$.
Then the resulting  ``generic'' dCHA $\sh(Q)$ is split semisimple. Now
 assume  that $\theta :  \mathbb{Z}[q_1, \dots, q_{m}] \rightarrow \mathbb{K}$ is a specialization
     and let $\mathbb{K}\sh(Q)$ be the specialized algebra, where $\mathbb{K}$ is any field.
     Note that for all $\bgl\in \mpn$, $s_{\bgl}(Q)\in \mathbb{Z}[q_1, \dots, q_m]$
according to Corollary~\ref{Cor:integral}. Then,  by \cite[Theorem 7.2.6]{Geck},
      $\mathbb{K}\sh(Q)$ is (split) semisimple if and only if, for all $\bgl\in \mpn$,
      $\theta (s_{\bgl} (Q))\neq 0$. From this,   we can deduce the following:

  \begin{theorem}[\cite{AMR}, Theorem~6.11]\label{ssimple}
Assume that $\mathbb{K}$ is a field. The algebra  $\mathbb{K}\sh(Q)$ is (split) semi-simple if and only if
$\theta(P_{\sh}(Q))\neq 0$, where $P_{\sh}(Q)$ is defined in Assumption~\ref{separated}.\end{theorem}
\begin{proof}
Assume first that $\theta (P_{\sh}(Q))=0 $. There are three cases:
\begin{enumerate}
\item If $\theta(n!)=0$, then  $\theta (h^{\eta^1,\eta^1}_{1,n-i+1})=0$ for  $\boldsymbol{\eta}=\big((n);(0);\dots;(0)\big)\in \mpn$. Thus, for this $m$-multipartition, we have $\theta (s_{\bgl}(Q))=0$, which implies that  $\mathbb{K} \sh(Q)$ is not semisimple.
\item If there exist $1\leq s < t \leq m$ and $0 \leq k <n $
  such that $\theta(k+q_s-q_t )=0$, then $\theta ( h_{1,n-k}^{\lambda^s,\lambda^t}+q_s-q_t)=0$
  for  $\bgl\in \mpn$ with $\lambda^s=(n)$, $\lambda^t=(0)$.
  Thus $\theta (s_{\bgl} (Q))=0$ and  $\mathbb{K} \sh(Q)$ is not semisimple.
 \item If there exist $1\leq s < t \leq m$ and $-n<k<0$  such that $\theta (k+q_s-q_t )=0 $,
 then $\theta (h_{1,n+k}^{\lambda^t,\lambda^s}+q_t-q_s)=0$ for  $\bgl\in \mpn$  with $\lambda^s=(0)$, $\lambda^t=(n)$. Again, $\theta (s_{\bgl} (Q))=0$ and  $\mathbb{K} \sh(Q)$ is not semisimple.
 \end{enumerate}
Conversely, if  $\mathbb{K} \sh(Q)$ is not semisimple, then  there exists $\bgl  \in \mpn$ such that
$\theta (s_{\bgl}(Q))= 0$. As for all $1\leq s,t\leq m$ and $(i,j)\in [\lambda^s]$, $-n<h_{i,j}^{\lambda^s,\lambda^t}<n$, we conclude that $\theta ( P_{\sh}(Q))=0$.
\end{proof}

\begin{point}{}* We end this paper by giving some remarks for the study for the dCHA.
 Our motivation is that statements that are regarded as theorems in the setting of the cyclotomic Hecke algebra are
 often adopted as statements in the setting of the dCHA, and  vice versa (see e.g.
 Brundan and  Kleshchev's works \cite{BK-Block, BK-Math.Z.}, Ariki, Mathas and Rui's
 work \cite[\S6]{AMR} and \cite{zhao}).

Recall from \cite[Definition 2.11]{BMM} that the generic degree of $W_{m,n}$ are certain ``spetsial" specializations
of the rational functions $s_{\boldsymbol{\eta}}(Q)/s_{\bgl}(Q)$ where  $s_{\boldsymbol{\eta}}(Q)$ is
the Schur elements associated to the trivial representation of the cyclotomic Hecke algebras,
  which are polynomial with rational coefficients. Moreover, for these specializations $s_{\boldsymbol{\eta}}(Q)$
  is equal to the Poincar\'{e} polynomial of the coinvariant algebra of the reflection representation of $W_{m,n}$.
  It would be interesting to know what are
  the degenerate ``spetsial" specializations for the dCHA algebra and study the properties of
  these degenerate ``spetsial" specializations.

  Note that  Chlouveraki \cite{C} defined the $a$-function and $A$-function attached to very irreducible characters of the cyclotomic Hecke algebra of type $G(m,1,n)$ by involving the Schur elements and showed these functions are constant on the Rouquier blocks of cyclotomic Hecke algebra.  It is very possible that we can do the same issues for  the dCHA. We hope to return to these issues in future work.

 Finally note that Hu and Mathas \cite{Hu-Mathas} have shown that the (degenerate) cyclotomic Hecke algebras are graded symmetric algebras, it may be interesting to determine the explicit formula for the graded Schur elements of the (degenerate) cyclotomic Hecke algebras.
\end{point}

%%%%%%%%%%%%%%%%%%%%%%%%%%%%%%%%%%%%%%%%%%%%%%%%%%%%%%%%%%%%%%%%%%%%%%%%%
%%%%%%%%%%%%%%%%%%%%%%%%%%%%%%%%%%%%%%%%%%%%%%%%%%%%%%%%%%%%%%%%%%%%%%

\end{CJK*}
\end{document}